\documentclass[twoside, 11pt]{article}
\usepackage{amssymb, amsmath, mathrsfs, amsthm}%, wasysym}
\usepackage{graphicx}
\usepackage{color}
\usepackage[top=3cm, bottom=3cm, left=3cm, right=3cm]{geometry}
\usepackage{float, caption, subcaption}

\input{epsf}

\newcommand{\bd}{\begin{description}}
\newcommand{\ed}{\end{description}}
\newcommand{\bi}{\begin{itemize}}
\newcommand{\ei}{\end{itemize}}
\newcommand{\be}{\begin{enumerate}}
\newcommand{\ee}{\end{enumerate}}
\newcommand{\beq}{\begin{equation}}
\newcommand{\eeq}{\end{equation}}
\newcommand{\beqs}{\begin{eqnarray*}}
\newcommand{\eeqs}{\end{eqnarray*}}

\definecolor{DarkGreen}{rgb}{0.2, 0.6, 0.3}

\newtheorem{theorem}{Theorem}[section]

\newtheorem{lemma}{Lemma}[section]

\newtheorem{corollary}[theorem]{Corollary}

\newtheorem{claim}{Claim}
\newtheorem{remark}{Remark}[section]
\newtheorem{fact}{Fact}
\newtheorem{proposition}{Proposition}[section]

\newtheorem{observation}{Observation}[section]
\begin{document}
\title{\textbf{On the $g$-extra connectivity of graphs} \footnote{Supported by the National
Science Foundation of China (Nos. 11601254, 11551001, 11161037,
61763041, 11661068, and 11461054) and the Science Found of Qinghai
Province (Nos.  2016-ZJ-948Q, and 2014-ZJ-907) and the  Qinghai Key
Laboratory of Internet of Things Project (2017-ZJ-Y21).} }

\author{
Zhao Wang\footnote{College of Science, China Jiliang University,
Hangzhou 310018, China. {\tt wangzhao@mail.bnu.edu.cn}}, \ \ Yaping
Mao\footnote{Corresponding author} \footnote{School of Mathematics
and Statistis, Qinghai Normal University, Xining, Qinghai 810008,
China. {\tt maoyaping@ymail.com}}, \ \ Sun-Yuan Hsieh
\footnote{Department of Computer Science and Information
Engineering, National Cheng Kung University, Tainan 701, Taiwan {\tt
hsiehsy@mail.ncku.edu.tw}}}
\date{}
\maketitle

\begin{abstract}
Connectivity and diagnosability are two important parameters for the
fault tolerant of an interconnection network $G$. In 1996,
F\`{a}brega and Fiol proposed the $g$-extra connectivity of $G$. A
subset of vertices $S$ is said to be a \emph{cutset} if $G-S$ is not
connected. A cutset $S$ is called an \emph{$R_g$-cutset}, where $g$
is a non-negative integer, if every component of $G-S$ has at least
$g+1$ vertices. If $G$ has at least one $R_g$-cutset, the
\emph{$g$-extra connectivity} of $G$, denoted by $\kappa_g(G)$, is
then defined as the minimum cardinality over all $R_g$-cutsets of
$G$. In this paper, we first obtain the exact values of $g$-extra
connectivity of some special graphs. Next, we show that $1\leq
\kappa_g(G)\leq n-2g-2$ for $0\leq g\leq \left\lfloor
\frac{n-3}{2}\right\rfloor$, and graphs with $\kappa_g(G)=1,2,3$ and
trees with $\kappa_g(T_n)=n-2g-2$ are characterized, respectively.
In the end,
we get the three extremal results for the $g$-extra connectivity. \\[2mm]
\end{abstract}

\section{Introduction}

For a graph $G$, let $V(G)$, $E(G)$, $e(G)$, $\overline{G}$, and
$diam(G)$ denote the set of vertices, the set of edges, the size,
the complement, and the diameter of $G$, respectively. A subgraph
$H$ of $G$ is a graph with $V(H)\subseteq V(G)$, $E(H)\subseteq
E(G)$, and the endpoints of every edge in $E(H)$ belonging to
$V(H)$. For any subset $X$ of $V(G)$, let $G[X]$ denote the subgraph
induced by $X$; similarly, for any subset $F$ of $E(G)$, let $G[F]$
denote the subgraph induced by $F$. We use $G-X$ to denote the
subgraph of $G$ obtained by removing all the vertices of $X$
together with the edges incident with them from $G$; similarly, we
use $G-F$ to denote the subgraph of $G$ obtained by removing all the
edges of $F$ from $G$. If $X=\{v\}$ and $F=\{e\}$, we simply write
$G-v$ and $G-e$ for $G-\{v\}$ and $G-\{e\}$, respectively. For two
subsets $X$ and $Y$ of $V(G)$ we denote by $E_G[X,Y]$ the set of
edges of $G$ with one end in $X$ and the other end in $Y$. If
$X=\{x\}$, we simply write $E_G[x,Y]$ for $E_G[\{x\},Y]$. The {\it
degree}\index{degree} of a vertex $v$ in a graph $G$, denoted by
$deg_G(v)$, is the number of edges of $G$ incident with $v$. Let
$\delta(G)$ and $\Delta(G)$ be the minimum degree and maximum degree
of the vertices of $G$, respectively. The set of neighbors of a
vertex $v$ in a graph $G$ is denoted by $N_G(v)$. The {\it union}
$G\cup H$ of two graphs $G$ and $H$ is the graph with vertex set
$V(G)\cup V(H)$ and edge set $E(G)\cup E(H)$. If $G$ is the disjoint
union of $k$ copies of a graph $H$, we simply write $G=kH$. The
\emph{connectivity} $\kappa(G)$ of a graph $G$ is the minimum number
of vertices whose removal results in a disconnected graph or only
one vertex left.

With the rapid development of VLSI technology, a multiprocessor
system may contain hundreds or even thousands of nodes, and some of
them may be faulty when the system is implemented. As the number of
processors in a system increases, the possibility that its
processors may be comefaulty also increases. Because designing such
systems without defects is nearly impossible, reliability and fault
tolerance are two of the most critical concerns of multiprocessor
systems \cite{XuWangWang}.

By the definition proposed by Esfahanian \cite{Esfahanian}, a
multiprocessor system is fault tolerant if it can remain functional
in the presence of failures. Two basic functionality criteria have
received considerable attention. The first criterion for a system to
be regarded as functional is whether the network logically contains
a certain topological structure. This is the problem that occurs
when embedding one architecture into another \cite{Leighton, Xu}.
This approach involves using system-wide redundancy and
reconfiguration. The second functionality criterion considers a
multiprocessor system functional if a fault-free communication path
exists between any two fault-free nodes; that is, the topological
structure of the multiprocessor system remains connected in the
presence of certain failures. Thus, connectivity and edge
connectivity are two major measurements of this criterion \cite{Xu}.
The \emph{connectivity} of a graph $G$, denoted by $\kappa(G)$, is
the minimal number of vertices whose removal from produces a
disconnected graph or only one vertex; the \emph{edge connectivity}
of a graph $G$, denoted by $\lambda(G)$, is the minimal number of
edges whose removal from produces a disconnected graph. However,
these two parameters tacitly assume that all vertices that are
adjacent to, or all edges that are incident to, the same vertex can
potentially fail simultaneously. This is practically impossible in
some network applications. To address this deficiency, two specific
terms forbidden faulty set and forbidden faulty edge set are
introduced. The vertices in a forbidden faulty set or the edges in a
forbidden faulty edge set cannot fail simultaneously.

The $g$-extra connectivity has been an object of interest for many
years, and it was firstly introduced by F\`{a}brega and Fiol
\cite{FabregaFiol}. A subset of vertices $S$ is said to be a
\emph{cutset} if $G-S$ is not connected. A cutset $S$ is called an
\emph{$R_g$-cutset}, where $g$ is a non-negative integer, if every
component of $G-S$ has at least $g+1$ vertices. If $G$ has at least
one $R_g$-cutset, the \emph{$g$-extra connectivity} of $G$, denoted
by $\kappa_g(G)$, is then defined as the minimum cardinality over
all $R_g$-cutsets of $G$. Clearly, $\kappa_0(G) =\kappa(G)$ for any
connected non-complete graph $G$. So the $g$-extra connectivity can
be viewed as a generalization of the traditional connectivity, and
it can more accurately evaluate the reliability and fault tolerance
for large-scale parallel processing systems accordingly. For more
research on $g$-extra connectivity, we refer to
\cite{ChangTsaiHsieh, FabregaFiol, GuHao, GuHaoLiu, RenWang, WangMa,
WangWangWang, XuLinZhouHsieh, ZhangZhou, Zhou}.

The monotone property of $\kappa_g(G)$ for non-negative integer $g$
is true.
\begin{proposition}\label{pro1-1}
Let $g$ be a non-negative integer, and let $G$ be a connected graph.
Then
$$
\kappa_{g}(G)\leq \kappa_{g+1}(G)
$$
\end{proposition}
\begin{proof}
From the definition of $\kappa_{g+1}(G)$, there exist $X\subseteq
V(G)$ and $|X|=\kappa_{g+1}(G)$ such that each connected component
of the resulting graph has at least $g+2$ vertices. Clearly, each
connected component has at least $g+1$ vertices, and hence
$\kappa_{g}(G)\leq \kappa_{g+1}(G)$.
\end{proof}

The monotone property of $\kappa_0(G)$ is true in terms of connected
graphs $G$.
\begin{observation}\label{obs1-2}
Let $G$ be a connected graph. If $H$ is a spanning subgraph of $G$,
then $\kappa_0(H)\leq \kappa_0(G)$.
\end{observation}

But for $g\geq 1$, the above monotone property is not true.
\begin{remark}\label{rem1-1}
Let $G$ be a graph obtained from four cliques $X_1,X_2,Y_1,Y_2$ with
$|V(X_i)|\geq g+1 \ (i=1,2)$ and $|V(Y_j)|\geq g+1 \ (j=1,2)$ and
three vertices $u,v,w$ by adding edges in $E_G[u,X_1]\cup
E_G[u,X_2]\cup E_G[u,Y_1]\cup E_G[u,Y_2]\cup E_G[v,Y_1]\cup
E_G[v,Y_2]\cup E_G[w,Y_1]\cup E_G[w,Y_2]\cup \{uv,uw\}$. Let $H$ be
a graph obtained from $G$ by deleting all edges in $E(X_1)\cup
E(X_2)\cup E_G[u,Y_1]\cup \{uv,uw\}$. Clearly, $H$ is a spanning
subgraph of $G$; see Figure $1$. We first show that $\kappa_g(G)=1$.
By deleting the vertex $u$, there are three components and each of
them contains at least $g+1$ vertices, and hence $\kappa_g(G)\leq
1$, and so $\kappa_g(G)=1$. Next, we show that $\kappa_g(H)=2$. By
deleting the vertices $v,w$, there are two components and each of
them contains at least $g+1$ vertices, and hence $\kappa_g(H)\leq
2$. Note that $u$ is the unique cut vertex in $H$. By deleting $u$,
there are isolated vertices in $X_1\cup X_2$, and hence
$\kappa_g(H)\geq 2$. So $\kappa_g(H)=2>\kappa_g(G)$.
\end{remark}
\begin{figure}[!hbpt]
\begin{center}
\includegraphics[scale=0.7]{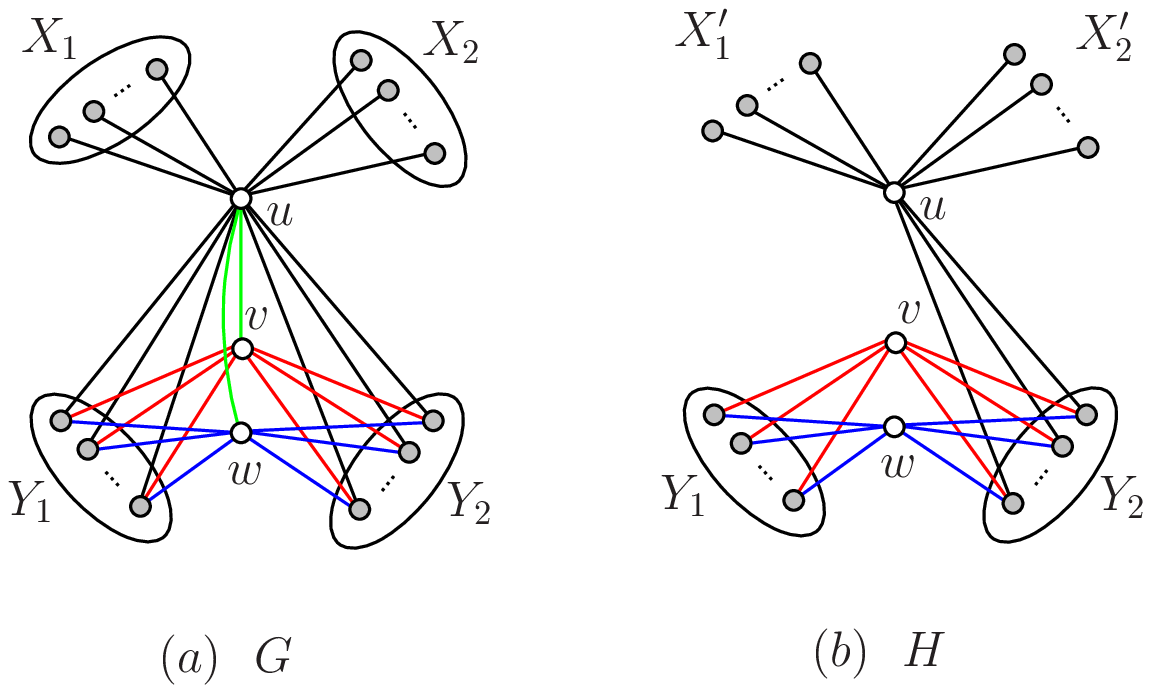}
\end{center}
\begin{center}
Figure 1: Graphs for Remark 1.1.
\end{center}\label{fig1}
\end{figure}

The range of the integer $g$ can be determined immediately.
\begin{proposition}\label{pro1-3}
Let $g$ be a non-negative integer. If $G$ has its $g$-extra
connectivity, then
$$
0\leq g\leq \left\lfloor \frac{n-3}{2}\right\rfloor
$$
and
$$
e(G)\leq {n\choose 2}-(g+1)^2.
$$
\end{proposition}
\begin{proof}
From the definition of $g$-extra connectivity, we delete at least
one vertex, and the resulting graph has at least two connected
components, and each connected component has at least $g+1$
vertices. Then $n-1\geq 2(g+1)$, and hence $0\leq g\leq \lfloor
\frac{n-3}{2}\rfloor$.

Since $\kappa_g(G)$ exists, it follows that there exists $X\subseteq
V(G)$ with $|X|=\kappa_g(G)$ such that $G-X$ is not connected and
each connected component of $G-X$ has at least $g+1$ vertices. Let
$C_1,C_2,\ldots,C_r$ be the connected components of $G-X$. Then
$|V(C_i)|\geq g+1$ for $i=1,2$. Since there is no edges from $C_1$
to $C_2$, it follows that $e(\overline{G})\geq
|V(C_1)||V(C_2)|=(g+1)^2$, and hence $e(G)\leq {n\choose
2}-(g+1)^2$.
\end{proof}

We consider the following problems in which their solutions will
give insights in designing interconnection networks with respect to
the size of the networks and the targeted $g$-extra connectivity.
\begin{itemize}
\item[] \noindent {\bf Problem 1.} Given two positive integers $n$ and
$k$, compute the minimum integer $s(n,k)=\min\{|E(G)|:G\in
\mathscr{G}(n,k)\}$, where $\mathscr{G}(n,k)$ the set of all graphs
of order $n$ (that is, with $n$ vertices) with $g$-extra
connectivity $k$.

\item[] \noindent {\bf Problem 2.} Given two positive integers $n$ and $k$,
compute the minimum integer $f(n,k)$ such that for every connected
graph $G$ of order $n$, if $|E(G)|\geq f(n,k)$ then $\kappa_g(G)\geq
k$.

\item[] \noindent {\bf Problem 3.} Given two positive integers $n$ and $k$,
compute the maximum integer $g(n,k)$ such that for every graph $G$
of order $n$, if $|E(G)|\leq g(n,k)$ then $\kappa_g(G)\leq k$.
\end{itemize}

In Section $2$, we first obtain the exact values of $g$-extra
connectivities of complete bipartite graphs, complete multipartite
graphs, joined graphs and corona graphs. For a connected graph $G$
of order $n$, we show that $\kappa_g(G)\leq n-diam(G)$ for $0\leq
g\leq \left\lfloor\frac{diam(G)}{2}\right\rfloor-1$, and $1\leq
\kappa_g(G)\leq n-2g-2$ for $0\leq g\leq \left\lfloor
\frac{n-3}{2}\right\rfloor$ in Section $3$. Graphs with
$\kappa_g(G)=1,2,3$ and trees with $\kappa_g(T_n)=n-2g-2$ are
characterized, respectively, in Section $4$. In the end, we get the
extremal results for the $g$-extra connectivity in Section $5$.

\section{Results for special graphs}

In this section, we obtain the exact values for $g$-extra
connectivity of some special graphs.
\begin{proposition}\label{pro2-2}
Let $g$ be a non-negative integer.

$(1)$ If $K_{a,b} \ (a\geq b\geq 2)$ is a complete bipartite graph,
then $g=0$ and $\kappa_g(K_{a,b})=b$.

$(2)$ Let $r$ be an integer with $r\geq 3$. For complete
multipartite graph $K_{n_1,n_2,\ldots,n_r}$ $(n_1\leq n_2\leq \ldots
\leq n_r)$, we have $g=0$ and
$$
\kappa_g(K_{n_1,n_2,\ldots,n_r})=\sum_{i=1}^{r-1}n_i.
$$
\end{proposition}
\begin{proof}
$(1)$ By deleting any vertex in $K_{a,b}$, the resulting graph is
still a complete bipartite graph and it is connected. If the
resulting graph is not connected, then we must delete all the
vertices of one part. Then $g=0$. Since $a\geq b\geq 2$, we have
$\kappa_g(K_{a,b})=b$.

$(2)$ Similarly to the proof of $(1)$, we can get
$\kappa_g(K_{n_1,n_2,\ldots,n_r})=\sum_{i=1}^{r-1}n_i$.
\end{proof}

The {\it join} or {\it complete product} of two disjoint graphs $G$
and $H$, denoted by $G\vee H$, is the graph with vertex set
$V(G)\cup V(H)$ and edge set $E(G)\cup E(H)\cup \{uv\,|\, u\in V(G),
v\in V(H)\}$.
\begin{theorem}\label{th2-1}
Let $n,m,g$ be three non-negative integers with $0\leq g\leq \lfloor
\frac{n-3}{2}\rfloor$ and $0\leq g\leq \lfloor
\frac{m-3}{2}\rfloor$. Let $G,H$ be two connected graph of order
$n,m$, respectively.

$(1)$ If $0\leq g\leq \min\{\lfloor \frac{n-3}{2}\rfloor,\lfloor
\frac{m-3}{2}\rfloor\}$, then
$$
\kappa_g(G\vee H)=\min\{\kappa_g(G)+|V(H)|,\kappa_g(H)+|V(G)|\}.
$$

$(2)$ If $\lfloor \frac{m-3}{2}\rfloor<g\leq \lfloor
\frac{n-3}{2}\rfloor$, then $\kappa_g(G\vee H)=\kappa_g(G)+|V(H)|$.
\end{theorem}
\begin{proof}
$(1)$ From the definition of $\kappa_g(G)$, there exists $X\subseteq
V(G)$ with $|X|=\kappa_g(G)$ such that $G-X$ is not connected and
each connected component of $G-X$ has at least $g+1$ vertices. Let
$S=X\cup V(H)$. Then $S\subseteq V(G\vee H)$ and $G\vee H-S=G-X$ is
not connected and each connected component of $G\vee H-S$ has at
least $g+1$ vertices, and hence $\kappa_g(G\vee H)\leq
|S|=|X|+|V(H)|=\kappa_g(G)+|V(H)|$. Similarly, $\kappa_g(G\vee
H)\leq \kappa_g(H)+|V(G)|$. Therefore, we have $\kappa_g(G\vee
H)\leq \min\{\kappa_g(G)+|V(H)|,\kappa_g(H)+|V(G)|\}$.

From the definition of $\kappa_g(G\vee H)$, there exists $S\subseteq
V(G\vee H)$ with $|S|=\kappa_g(G\vee H)$ such that $(G\vee H)-S$ is
not connected and each connected component of $(G\vee H)-S$ has at
least $g+1$ vertices. Since $(G\vee H)-S$ is not connected, it
follows from the structure of $G\vee H$ that $V(G)\subseteq S$ or
$V(H)\subseteq S$. Without loss of generality, let $V(H)\subseteq
S$. Let $S'=S-V(H)$. From the definition of $\kappa_g(G)$, $|S'|\geq
\kappa_g(G)$ and hence $\kappa_g(G\vee H)=|S|=|S'|+|V(H)|\geq
\kappa_g(G)+|V(H)|$. So we have $\kappa_g(G\vee H)\geq
\min\{\kappa_g(G)+|V(H)|,\kappa_g(H)+|V(G)|\}$.

$(2)$ It follows from $(1)$.
\end{proof}

\begin{remark}\label{rem2-1}
Suppose $g>\max\{\lfloor \frac{n-3}{2}\rfloor,\lfloor
\frac{m-3}{2}\rfloor\}$. From the definition of $\kappa_g(G\vee H)$,
there exists $S\subseteq V(G\vee H)$ with $|S|=\kappa_g(G\vee H)$
such that $G\vee H-S$ is not connected and each connected component
of $G\vee H-S$ has at least $g+1$ vertices. Since $G\vee H-S$ is not
connected, it follows that $V(G)\subseteq S$ or $V(H)\subseteq S$.
Without loss of generality, let $V(H)\subseteq S$. Let
$S-V(H)\subseteq V(G)$. From Proposition \ref{pro1-3}, we have
$0\leq g\leq \lfloor \frac{n-3}{2}\rfloor$, a contradiction. So
$\kappa_g(G\vee H)$ does not exist if $g>\max\{\lfloor
\frac{n-3}{2}\rfloor,\lfloor \frac{m-3}{2}\rfloor\}$.
\end{remark}

\subsection{Corona}

The {\it corona} $G*H$ is obtained by taking one copy of $G$ and
$|V(G)|$ copies of $H$, and by joining each vertex of the $i$-th
copy of $H$ with the $i$-th vertex of $G$, where
$i=1,2,\ldots,|V(G)|$. The corona graphs was introduced by Frucht
and Harary \cite{FruchtHarary}. For more details on corona graphs,
we refer to \cite{KuziakYero, LiuZhou, YeroKuziak}.

In $G*H$, let $G$ and $H$ be two graphs with
$V(G)=\{u_1,u_2,\ldots,u_{n}\}$ and $V(H)=\{v_1,v_2,\ldots,v_{m}\}$,
respectively. From the definition of corona graphs, $V(G*H)=V(G)\cup
\{(u_i,v_j)\,|\,1\leq i\leq n, \ 1\leq j\leq m\}$, where $*$ denotes
the corona product operation. For $u\in V(G)$, we use $H(u)$ to
denote the subgraph of $G*H$ induced by the vertex set
$\{(u,v_j)\,|\,1\leq j\leq m\}$. For fixed $i \ (1\leq i\leq n)$, we
have $u_i(u_i,v_j)\in E(G*H)$ for each $j \ (1\leq j\leq m)$. Then
$V(G*H)=V(G)\cup V(H(u_1))\cup V(H(u_2))\cup \ldots \cup V(H(u_n))$.
\begin{theorem}\label{th2-2}
Let $G$ be a connected graph of order at least $2$, and let $H$ be a
connected graph of order $m$.

$(1)$ If $0\leq g\leq m-1$, then $\kappa_g(G*H)=1$.

$(2)$ If $k(m+1)<g+1\leq (k+1)(m+1)$ where $1\leq k\leq
\lfloor\frac{n-3}{2}\rfloor$, then
$$
\kappa_g(G*H)=|X|(m+1),
$$
where $X$ is a minimum vertex subset of $G$ such that the resulting
graph of $G-X$ is not connected and each connected component has at
least $(k+1)$ vertices.
\end{theorem}
\begin{proof}
$(1)$ Note that $V(G)=\{u_1,u_2,\ldots,u_{n}\}$. It is clear that
$G*H-u_1$ is not connected. Since $0\leq g\leq m-1$, it follows that
each of connected components of $G*H-u_1$ has at least $g+1$
vertices, and hence $\kappa_g(G*H)\leq 1$. Since $G*H$ is connected,
it follows that $\kappa_g(G*H)=1$.

$(2)$ From the definition of $X$, $G-X$ is not connected and each
connected component of $G-X$ has at least $k+1$ vertices. Let
$|X|=x$. Without loss of generality, let
$X=\{u_{i_1},u_{i_2},\ldots,u_{i_{x}}\}$. Let
$$
S=X\cup V(H(u_{i_1}))\cup V(H(u_{i_2}))\cup \ldots \cup
V(H(u_{i_{x}})).
$$
Then $|S|=(m+1)x$. Clearly, $G*H-S$ is not connected and each
connected component of $G*H-S$ has at least $(k+1)(m+1)$ vertices,
and hence $\kappa_g(G*H)\leq |S|=(m+1)x=|X|(m+1)$.

It suffices to show that $\kappa_g(G*H)\geq |X|(m+1)$. From the
definition of $\kappa_g(G*H)$, there exists $Y\subseteq V(G*H)$ with
$|Y|=\kappa_g(G*H)$ such that $G*H-Y$ is not connected and each
connected component of $G*H-Y$ has at least $g+1$ vertices.
\begin{claim}\label{Clm:1}
$|Y\cap V(G)|\geq |X|$.
\end{claim}
\begin{proof}
Assume, to the contrary, that $|Y\cap V(G)|<|X|$. Then $G-Y$ is not
connected, and there exists a connected component, say $C_1$, of
$G-Y$ has at most $k$ vertices. Let
$V(C_1)=\{u_{j_1},u_{j_2},\ldots,u_{j_y}\}$, where $y\leq k$. Then
$C_1 \cup H(u_{j_1})\cup H(u_{j_2})\cup \ldots \cup H(u_{j_{y}})$ is
a connected component of $G*H-Y$ having at most $k(m+1)$ vertices,
which contradicts to the fact that $g+1>k(m+1)$, a contradiction.
\end{proof}

From Claim \ref{Clm:1}, we have $|Y\cap V(G)|\geq |X|$. Let $Y\cap
V(G)=\{u_{a_1},u_{a_2},\ldots,u_{a_z}\}$. Then $H(u_{a_i})\subseteq
Y$ for each $i \ (1\leq i\leq z)$, and hence $\kappa_g(G*H)=|Y|\geq
|Y\cap V(G)|(m+1)\geq |X|(m+1)$.
\end{proof}

\section{Upper and lower bounds}

The following upper and lower bounds are immediate.
\begin{proposition}\label{pro3-1}
Let $g$ be a non-negative integer and let $G$ be a connected graph
of order $n$ such that $0\leq g\leq \left\lfloor
\frac{n-\kappa(G)-2}{2}\right\rfloor$. Then
$$
\kappa(G)\leq \kappa_g(G)\leq n-2g-2.
$$
Moreover, the upper and lower bounds are sharp.
\end{proposition}
\begin{proof}
From Proposition \ref{pro3-1}, we have $\kappa_g(G)\geq
\kappa_0(G)=\kappa(G)$. Suppose $\kappa_g(G)\geq n-2g-1$. From the
definition of $\kappa_g(G)$, there exists $X\subseteq V(G)$ and
$|X|=\kappa_g(G)$ such that there are at least two components and
one of them has no more than $g$ vertices, a contradiction. So
$\kappa(G)\leq \kappa_g(G)\leq n-2g-2$. Theorem \ref{th4-3} shows
that the upper bound is sharp. If $k=0$, then
$\kappa(G)=\kappa_0(G)$. This implies that the lower bound is sharp.
\end{proof}

The following corollary is immediate from Proposition \ref{pro3-1}.
\begin{corollary}\label{cor3-1}
Let $n,g$ be two integers with $0\leq g\leq \left\lfloor
\frac{n-3}{2}\right\rfloor$. If $G$ is a connected graph of order
$n$, then
$$
1\leq \kappa_g(G)\leq n-2g-2.
$$
Moreover, the upper and lower bounds are sharp.
\end{corollary}

\begin{proposition}\label{pro2-3}
Let $g$ be a non-negative integer.

$(1)$ If $W_{n} \ (n\geq 5)$ is a wheel of order $n$, then $0\leq
g\leq \lfloor \frac{n-5}{2}\rfloor$ and $\kappa_g(W_{n})=3$.

$(2)$ If $P_{n} \ (n\geq 3)$ be a path of order $n$, then $0\leq
g\leq \lfloor \frac{n-3}{2}\rfloor$ and $\kappa_g(P_{n})=1$.
\end{proposition}
\begin{proof}
$(1)$ From Proposition \ref{pro3-1}, we have $\kappa_g(W_{n})\geq
\kappa(W_{n})=3$. It suffices to show that $\kappa_g(W_{n})\leq 3$.
Let $v$ be the center of $W_{n}$, and $W_{n}-v=C_{n-1}$, and
$V(C_{n-1})=\{u_1,u_2,\ldots,u_{n-1}\}$. Choose
$X=\{v,u_1,u_{\lfloor\frac{n-1}{2}\rfloor}\}$. Clearly, $n-3\geq
2(g+1)$, that is, $0\leq g\leq \lfloor \frac{n-5}{2}\rfloor$. Since
each component of $W_{n}-X$ has $g+1$ vertices, it follows that
$\kappa_g(W_{n})\leq 3$, and hence $\kappa_g(W_{n})=3$.

$(2)$ From Proposition \ref{pro3-1}, we have $\kappa_g(P_{n})\geq
\kappa(P_{n})=1$. It suffices to show $\kappa_g(P_{n})\leq 1$. Let
$P_n=u_1u_2\ldots u_{n}$. Choose $v=u_{\lceil n/2\rceil}$. Clearly,
$n-1\geq 2(g+1)$, that is, $0\leq g\leq \lfloor
\frac{n-3}{2}\rfloor$. Then each component of $G-v$ has $g+1$
vertices, and hence $\kappa_g(P_{n})\leq 1$. So $\kappa_g(P_{n})=1$.
\end{proof}

In terms of diameter, we can get upper bound of $\kappa_g(G)$.
\begin{proposition}\label{pro3-2}
Let $g$ be a non-negative integer and let $G$ be a connected graph
of order $n$ such that $0\leq g\leq
\left\lfloor\frac{diam(G)}{2}\right\rfloor-1$. Then
$$
\kappa_g(G)\leq n-diam(G).
$$
Moreover, the bound is sharp.
\end{proposition}
\begin{proof}
Let $diam(G)=d$ and let $P_{d+1}=v_1v_2\ldots v_{d+1}$ be a path
with $diam(P_{d+1})=d$. Let $X=V(G)-V(P_{d+1})$. Then $|X|=n-d-1$
and $G-X=P_{d+1}$. Choose $v=u_{\lceil d/2\rceil}$. Since $0\leq
g\leq \left\lfloor\frac{diam(G)}{2}\right\rfloor-1$, it follows that
each component of $G-X-v$ has $g+1$ vertices, and hence
$\kappa_g(G)\leq |X|+1=n-d-1+1=n-d$.

To show the sharpness of this bound, we consider the path $P_n$.
From Proposition \ref{pro2-3}, we have $\kappa_g(P_{n})=1$ for
$0\leq g\leq \lfloor \frac{n-1}{2}\rfloor-1$. Clearly,
$diam(P_n)=n-1$ and $\kappa_g(P_n)=1=n-diam(P_n)$.
\end{proof}

\section{Graphs with given $g$-extra connectivity}

From Corollary \ref{cor3-1}, for $0\leq g\leq \left\lfloor
\frac{n-3}{2}\right\rfloor$, we have $1\leq \kappa_g(G)\leq n-2g-2$.

\subsection{Graphs with large $g$-extra connectivity}

The following observation is immediate for graphs with
$\kappa_g(G)=n-2g-2$ and $g=0$.
\begin{observation}\label{obs4-2}
Let $n,g$ be two integers with $0\leq g\leq \left\lfloor
\frac{n-3}{2}\right\rfloor$, and let $G$ be a connected graph of
order $n$. Then $\kappa_g(G)=n-2g-2$ and $g=0$ if and only if $G$ is
a graph obtained from $K_n$ by deleting a matching.
\end{observation}

It seems that it is not easy to characterize graphs with
$\kappa_g(G)=n-2g-2$ for general graph $G$. So we focus our
attention on trees.

Let $T_n^*$ be a tree of order $n$ constructed as follows:
\begin{itemize}
\item[] $(1)$ Let $T'$ and $T''$ be two trees with
$|V(T')|=|V(T'')|=g+1$;

\item[] $(2)$ Let $T_1,T_2,\ldots, T_r$ be trees with
$|V(T_i)|\leq g$ for each $1\leq i\leq r$ and
$\sum_{i=1}^r|V(T_i)|=n-2g-3$;

\item[] $(3)$ Let $T_n^*$ be a tree of order $n$ obtained from the subtrees
$T',T'',T_1,T_2,\ldots, T_r$ by adding a new vertex $v$, and then
adding one edge from $v$ to each $T_i$ and one edge from $v$ to $T'$
and $T''$, respectively, where $r\geq 0$ and $1\leq i\leq r$.
\end{itemize}

\begin{lemma}\label{lem4-1}
Let $n,g$ be two integers with $1\leq g\leq \left\lfloor
\frac{n-3}{2}\right\rfloor$, and let $T_n^*$ be a tree of order $n$.
Then
$$
\kappa_g(T_n^*)=n-2g-2.
$$
\end{lemma}
\begin{proof}
From Corollary \ref{cor3-1}, we have $\kappa_g(T_n^*)\leq n-2g-2$.
It suffices to show that $\kappa_g(T_n^*)\geq n-2g-2$. We only need
to prove that for any $X\subseteq V(T_n^*)$ with $|X|\leq n-2g-3$,
$T_n^*-X$ is connected, or $T_n^*-X$ is not connected and there
exists a component of $T_n^*-X$ having at most $g$ vertices. If
$T_n^*-X$ is connected, then the result follows. Suppose that
$T_n^*-X$ is not connected. Let $C_1,C_2,\ldots,C_r$ be the
components of $T_n^*-X$. Then $v\in \bigcup_{i=1}^rV(C_i)$ or $v\in
X$.

Suppose $v\in \bigcup_{i=1}^rV(C_i)$. Without loss of generality,
let $v\in V(C_1)$. Suppose $C_2$ is a subtree of $T'$. If $C_2=T'$,
then $v\in V(C_2)$, which contradicts to the fact $v\in V(C_1)$.
Therefore, $C_2$ is a subtree of $T'$ and $|V(C_2)|\leq |V(T')|-1$.
Clearly, $C_2$ has at most $g$ vertices, as desired. The same is
true for the case that $C_2$ is a subtree of $T''$. Suppose $C_2$ is
neither a subtree of $T'$ nor a subtree of $T''$. Since $v\notin X$,
we can assume that $C_2$ is a subtree of $T_i \ (1\leq i\leq r)$.
Clearly, $C_2$ has at most $g$ vertices.

We now assume $v\in X$. Since $|V(T_i)|\leq g$ for each $i \ (1\leq
i\leq r)$, it follows that there exists a vertex $w$ in
$\bigcup_{i=1}^rV(T_i)$ such that $w\notin X$. The connected
component containing $w$ in $G-X$ has at most $g$ vertices.

From the above argument, we have $\kappa_g(T_n^*)\geq n-2g-2$, and
hence $\kappa_g(T_n^*)=n-2g-2$.
\end{proof}

Trees with $\kappa_g(T_n)=n-2g-2$ for general $g \ (0\leq g\leq
\left\lfloor \frac{n-3}{2}\right\rfloor)$ can be characterized.
\begin{theorem}\label{th4-3}
Let $n,g$ be two integers with $0\leq g\leq \left\lfloor
\frac{n-3}{2}\right\rfloor$, and let $T_n$ be a tree of order $n$.
Then $\kappa_g(T_n)=n-2g-2$ if and only if $T_n=T_n^*$.
\end{theorem}
\begin{proof}
Suppose $T_n=T_n^*$. From Lemma \ref{lem4-1}, we have
$\kappa_g(T_n^*)=n-2g-2$. Conversely, we suppose
$\kappa_g(T_n)=n-2g-2$. Then there exists $X\subseteq V(T_n)$ with
$|X|=n-2g-2$ such that there are two connected components $C_1,C_2$
with $|V(C_1)|=|V(C_2)|=g+1$. Then we have the following claim.

\begin{claim}\label{Clm:2}
There exists a vertex $v\in X$ such that there is an edge from $v$
to each $C_i$, $i=1,2$.
\end{claim}
\begin{proof}
Assume, to the contrary, that for any $v\in X$, there is no edge
from $v$ to $C_1$ or $C_2$. Without loss of generality, we suppose
that there is no edge from $v$. Then there exists a path
$vx_1x_2\ldots x_au$ where $u\in V(C_1)$ and there exists a path
$vy_1y_2\ldots y_bw$ where $w\in V(C_2)$. Clearly, $a\geq 1$ and
$y\geq 0$. Let
$$
Y=V(G)-V(C_1)-V(C_2)-\{x_1,x_2,\ldots, x_a\}-\{y_1,y_2,\ldots,
y_b\}.
$$
Then $|Y|\leq n-2g-3$ and $G-Y$ is not connected and each connected
component of $G-Y$ has at least $g+1$ vertices, which contradicts to
the fact $\kappa_g(T_n)=n-2g-2$.
\end{proof}

From Claim \ref{Clm:2}, there exists a vertex $v\in X$ such that
there is an edge from $v$ to each $C_i$, $i=1,2$. Let
$T_1,T_2,\ldots,T_r$ be the connected components of
$G-V(C_1)-V(C_2)-v$. Then we have the following claim.

\begin{claim}\label{Clm:3}
For each $T_i \ (1\leq i\leq r)$, there is no edges from $T_i$ to
$C_j \ (j=1,2)$.
\end{claim}
\begin{proof}
Assume, to the contrary, that there exists some $T_i$ such that
there exists an edge $u_iv_j$ from $T_i$ to $C_j$, where $u_i\in
V(T_i)$ and $v_j\in V(C_j)$. Let $Z=V(G)-V(C_1)-V(C_2)-u_i$. Then
$|Z|\leq n-2g-3$ and $G-Z$ is not connected and each connected
component of $G-Z$ has at least $g+1$ vertices, which contradicts to
the fact $\kappa_g(T_n)=n-2g-2$.
\end{proof}

From Claim \ref{Clm:3}, for each $T_i \ (1\leq i\leq r)$, there is
an edge from $T_i$ to $v$. Furthermore, we have the following claim.

\begin{claim}\label{Clm:4}
For each $T_i \ (1\leq i\leq r)$, $|T_i|\leq g$.
\end{claim}
\begin{proof}
Assume, to the contrary, that there exists some $T_j$ such that
$|T_j|\geq g+1$. Let $A=\{v\}\cup (\bigcup_{i=2}^rV(T_i))$. Then
$|A|\leq n-3g-3$ and $G-A$ is not connected and each connected
component of $G-A$ has at least $g+1$ vertices, which contradicts to
the fact $\kappa_g(T_n)=n-2g-2$.
\end{proof}

From Claim \ref{Clm:4}, we have $|T_i|\leq g$ for each $T_i \ (1\leq
i\leq r)$. Then $T_n=T_n^*$, as desired.
\end{proof}

\subsection{Graphs with small $g$-extra connectivity}

Graphs with $\kappa_g(G)=1$ can be characterized easily.
\begin{observation}\label{obs4-1}
Let $n,g$ be two integers with $0\leq g\leq \left\lfloor
\frac{n-3}{2}\right\rfloor$, and let $G$ be a connected graph of
order $n$. Then $\kappa_g(G)=1$ if and only if there exists a cut
vertex $v$ in $G$ such that each connected component of $G-v$ has at
least $g+1$ vertices.
\end{observation}

We can also characterize graphs with $\kappa_g(G)=2$.
\begin{theorem}\label{th4-1}
Let $n,g$ be two integers with $0\leq g\leq \left\lfloor
\frac{n-3}{2}\right\rfloor$, and let $G$ be a connected graph of
order $n$. Then $\kappa_g(G)=2$ if and only if $G$ satisfies one of
the following conditions.

$(1)$ $\kappa(G)=2$ and there exists a cut vertex set $\{u,v\}$ in
$G$ such that each connected component of $G-\{u,v\}$ has at least
$g+1$ vertices;

$(2)$ $\kappa(G)=1$, and $g\geq 1$,
\begin{itemize}
\item[] and $(a)$ for each cut vertex $u$, there exists a connected component
of $G-u$ having at most $g$ vertices,

\item[] and $(b)$ there exists a cut vertex $v$ such that $G-v$ contains at least
$3$ connected components, where one of the component is an isolated
vertex and each of the other components has at least $g+1$ vertices,
or there are two non-cut vertices $x,y$ such that $G-\{x,y\}$ is not
connected and each connected component has at least $g+1$ vertices.
\end{itemize}
\end{theorem}
\begin{proof}
Suppose that $G$ satisfies $(1)$ and $(2)$. Suppose that $(1)$
holds. Since $G-\{u,v\}$ has at least $g+1$ vertices, it follows
that $\kappa_g(G)\leq 2$. From Proposition \ref{pro3-1}, we have
$\kappa_g(G)\geq \kappa(G)=2$. Suppose that $(2)$ holds. Since for
each cut vertex $u$, there exists a connected component of $G-u$
having at most $g$ vertices, it follows that $\kappa_g(G)\geq 2$. If
there exists a cut vertex $v$ such that $G-v$ contains at least $3$
connected components, where one of the component is an isolated
vertex, say $u$, and each of the other components has at least $g+1$
vertices, then $G-u-v$ is not connected and each component has at
least $g+1$ vertices, and hence $\kappa_g(G)\leq 2$. If there are
two non-cut vertices $x,y$ such that $G-\{x,y\}$ is not connected
and each connected component has at least $g+1$ vertices, then
$\kappa_g(G)\leq 2$. So we have $\kappa_g(G)=2$.

Conversely, we suppose $\kappa_g(G)=2$. From Proposition
\ref{pro3-1}, we have $\kappa(G)\leq 2$. Suppose $\kappa(G)=2$. If
for each vertex cut set $\{u,v\}$ in $G$, there exists a connected
component of $G-\{u,v\}$ having at most $g$ vertices, then
$\kappa_g(G)\geq 3$, a contradiction. So there exists a vertex cut
set $\{u,v\}$ in $G$ such that each connected component of
$G-\{u,v\}$ has at least $g+1$ vertices, as desired.

Suppose $\kappa(G)=1$. Then we have the following claim.

\begin{claim}\label{Clm:5}
$g\geq 1$.
\end{claim}
\begin{proof}
Assume, to the contrary, that $g=0$. By deleting one cut vertex,
each connected component has at least one vertex, and hence
$\kappa_g(G)=1$, which contradicts to the fact $\kappa_g(G)=2$.
\end{proof}

From Claim \ref{Clm:5}, we have $g\geq 1$. Since $\kappa_g(G)=2$, we
have the following facts.
\begin{fact}\label{Fact:Fact1}
For any cut vertex $v$, there exists a connected component of $G-v$
having at most $g$ vertices.
\end{fact}

\begin{fact}\label{Fact:Fact2}
There exist two vertices $x,y$ in $G$ such that $G-\{x,y\}$ is not
connected and each connected component of $G-\{x,y\}$ has at least
$g+1$ vertices.
\end{fact}

Suppose that one of $x,y$ is a cut vertex of $G$. Without loss of
generality, we assume that $x$ is a cut vertex of $G$. Let
$C_1,C_2,\ldots,C_r$ be the connected components of $G-x$.

\begin{claim}\label{Clm:6}
$r\geq 3$.
\end{claim}
\begin{proof}
Assume, to the contrary, that $r=2$. From Fact \ref{Fact:Fact1},
there exists a connected component of $G-x$, say $C_1$, having at
most $g$ vertices. If $y\notin V(C_1)$, then $C_1$ has at most $g$
vertices in $G-\{x,y\}$, which contradicts to Fact \ref{Fact:Fact2}.
Suppose $y\in V(C_1)$. If $|V(C_1)|=1$, then $G-\{x,y\}$ contains
exactly one connected component, which contradicts to Fact
\ref{Fact:Fact2}. Suppose $|V(C_1)|\geq 2$. Since $C_1$ has at most
$g$ vertices in $G$, it follows that $C_1-y$ has at most $g-1$
vertices in $G-\{x,y\}$, which contradict to Fact \ref{Fact:Fact2}.
\end{proof}

From Fact \ref{Fact:Fact1}, we suppose that $C_1$ has at most $g$
vertices. From Fact \ref{Fact:Fact2}, we have $C_1=\{y\}$ and for
each $i \ (2\leq i\leq r)$, $C_i$ has at least $g+1$ vertices.
Clearly, $(2)$ holds.

Suppose that neither $x$ nor $y$ is a cut vertex of $G$. From Fact
\ref{Fact:Fact2}, $G-\{x,y\}$ is not connected and each connected
component has at least $g+1$ vertices.
\end{proof}

\noindent {\bf Example 4.1.} Let $H_1$ be a graph obtained from
$K_{\lceil \frac{n-2}{2}\rceil}$ and $K_{\lfloor
\frac{n-2}{2}\rfloor}$ by adding two vertices $u,v$ and edges in
$\{uu_1,uu_2,vv_1,vv_2\}$, where $u_1,u_2\in V(K_{\lceil
\frac{n-2}{2}\rceil})$ and $v_1,v_2\in V(K_{\lfloor
\frac{n-2}{2}\rfloor})$. From Theorem \ref{th4-1},
$\kappa(H_1)=\kappa_g(H_1)=2$.\\

\noindent {\bf Example 4.2.} Let $H_2$ be a graph obtained from
$K_{\lceil \frac{n-2}{2}\rceil}$ and $K_{\lfloor
\frac{n-2}{2}\rfloor}$ by adding two vertices $u,v$ and then the
edge $uv$, all the edges from $v$ to $K_{\lceil
\frac{n-2}{2}\rceil}$, and all the edges from $v$ to $K_{\lfloor
\frac{n-2}{2}\rfloor}$. From Theorem \ref{th4-1}, $\kappa(H_2)=1$
and $\kappa_g(H_2)=2$.\\

Furthermore, graphs with $\kappa_g(G)=3$ can be characterized in the
following.
\begin{theorem}\label{th4-3}
Let $n,g$ be two integers with $1\leq g\leq \left\lfloor
\frac{n-5}{2}\right\rfloor$. Let $G$ be a connected graph of order
$n$. Then $\kappa_g(G)=3$ if and only if $G$ satisfies one of the
following conditions.

$(1)$ $\kappa(G)=3$ and there exists a cut vertex set $\{u,v,w\}$ in
$G$ such that each connected component of $G-\{u,v,w\}$ has at least
$g+1$ vertices;

$(2)$ $\kappa(G)=2$, and $(a),(b)$ hold, where
\begin{itemize}
\item[] $(a)$ for any vertex cut set $\{u,v\}$, there exists a connected component
of $G-u-v$ having at most $g$ vertices,

\item[] $(b)$ there exists a vertex cut set $\{u,v\}$ such that $G-u-v$ contains at least
$3$ connected components, where one of the component is an isolated
vertex $x$ and $xu,xv\in E(G)$ and each of the other components has
at least $g+1$ vertices, or there are three vertices $x,y,z$ such
that $G-\{x,y\}$, $G-\{x,z\}$ and $G-\{y,z\}$ are connected and
$G-x-y-z$ is not connected and each connected component of $G-x-y-z$
has at least $g+1$ vertices.
\end{itemize}

$(3)$ $\kappa(G)=1$, and $g\geq 2$, and one element in
$\{(c)(d)(e),(c)(d)(f),(c)(d)(g),(c)(h)\}$ holds, where
\begin{itemize}
\item[] $(c)$ For each cut vertex $v$, there exists a connected component
of $G-v$ having at most $g$ vertices.

\item[] $(d)$ For any two vertices $x,y$, if $G-x-y$ is not connected, then there
exists a connected component having at most $g$ vertices.

\item[] $(e)$ There exists a cut vertex $v$ such that $G-v$ contains at least
$4$ connected components, and two of components has exactly one
isolated vertex, and each of other components has at least $g+1$
vertices.

\item[] $(f)$ There exists a cut vertex $v$ such that $G-v$ contains at least
$3$ connected components, and one of components has exactly two
vertices, and each of other components has at least $g+1$ vertices.

\item[] $(g)$ There exit two non-cut vertices $x,y$ such that $G-x-y$
contains at least $3$ connected components and one of connected
components has exactly one isolated vertex and each of other
components contains at least $g+1$ vertices.

\item[] $(h)$ There are three vertices $x,y,z$ such
that $G-\{x,y\}$, $G-\{x,z\}$ and $G-\{y,z\}$ are connected and
$G-x-y-z$ is not connected and each connected component of $G-x-y-z$
has at least $g+1$ vertices.
\end{itemize}
\end{theorem}
\begin{proof}
Suppose $\kappa_g(G)=3$. From Proposition \ref{pro3-1}, we have
$\kappa(G)\leq 3$. Suppose $\kappa(G)=3$. Then
$\kappa_g(G)=\kappa(G)=3$, and there there exists a vertex cut set
$\{u,v,w\}$ in $G$ such that each connected component of
$G-\{u,v,w\}$ has at least $g+1$ vertices.

Suppose $\kappa(G)=2$. If there exists a vertex cut set $\{u,v\}$
such that any connected component of $G-u-v$ has at least $g+1$
vertices, then $\kappa_g(G)\leq 2$, which contradicts to the fact
$\kappa_g(G)=3$. So for any cut vertex set $\{u,v\}$, there exists a
connected component of $G-u-v$ having at most $g$ vertices, i.e.,
$(a)$ holds. Since $\kappa_g(G)=3$, it follows that there there
exists a vertex cut set $\{u,v,w\}$ in $G$ such that each connected
component of $G-u-v-w$ has at least $g+1$ vertices. If $G-u-v$,
$G-w-v$, $G-u-w$ are all connected, then there are three vertices
$x,y,z$ such that $G-\{x,y\}$, $G-\{x,z\}$ and $G-\{y,z\}$ are
connected and $G-x-y-z$ is not connected and each connected
component of $G-x-y-z$ has at least $g+1$ vertices. If $G-u-v$ or
$G-w-v$ or $G-u-w$ is not connected, then we suppose that $G-u-v$ is
not connected, and hence $\{u,v\}$ is a cut vertex set. From $(a)$,
there exists a connected component of $G-u-v$, say $C_1$, having at
most $g$ vertices. Let $C_2,C_3,\ldots, C_r$ are other components of
$G-u-v$. We claim that each $C_i \ (2\leq i\leq r)$ contains at
least $g+1$ vertices. Assume, to the contrary, that $C_2$ contains
at most $g$ vertices. Then there is a connected component of
$G-u-v-w$ having at most $g$ vertices, which contradicts to the fact
$\kappa_g(G)=3$. Therefore, $C_i \ (2\leq i\leq r)$ contains at
least $g+1$ vertices. We claim that $|V(C_1)|=1$. Assume, to the
contrary, that $|V(C_1)|\geq 2$. Then there there is a connected
component having at most $g$ vertices in $G-\{u,v,w\}$, which
contradicts to the fact $\kappa_g(G)=3$. So $(b)$ holds.

Suppose $\kappa(G)=1$. Since $\kappa_g(G)=3$, it follows that $(c)$
and $(d)$ hold, and there exists a vertex cut set $\{u,v,w\}$ such
that each connected component of $G-\{u,v,w\}$ has at least $g+1$
vertices.

\begin{claim}\label{Clm:7}
At most two of $u,v,w$ are cut vertices in $G$.
\end{claim}
\begin{proof}
Assume, to the contrary, that $u,v,w$ are all cut vertices in $G$.
From $(c)$, there exists a connected component of $G-u$, say $C_1$,
having at most $g$ vertices. If $|V(C_1)|=1$, then $V(C_1)\neq
\{v\}$ or $V(C_1)\neq \{w\}$ since $v,w$ are cut vertices. Then
$C_1$ is an isolated vertex in $G-\{u,v,w\}$, a contradiction. If
$|V(C_1)|\geq 3$, then $C_1$ contains at most $g$ vertices
$G-\{u,v,w\}$. Clearly, $|V(C_1)|=2$ and $V(C_1)=\{v,w\}$, and one
of $v,w$ is not cut vertex of $G$, a contradiction.
\end{proof}

From Claim \ref{Clm:7}, at most two of $u,v,w$ are cut vertices in
$G$. Suppose that one of $u,v,w$, say $v$, is cut vertex in $G$.
From $(c)$, there exists a connected component of $G-v$, say $C_1$,
having at most $g$ vertices. Let $C_2,C_3,\ldots, C_r$ are other
components of $G-v$. Clearly, $|V(C_1)|\leq 2$. If $|V(C_1)|=2$,
then $(f)$ holds. If $|V(C_1)|=1$, then $V(C_1)=\{u\}$ or
$V(C_1)=\{w\}$. Without loss of generality, let $V(C_1)=\{u\}$. From
$(d)$, there exists a connected component of $G-u-v$ having at most
$g$ vertices, say $C_2$. Clearly, $V(C_2)=\{w\}$. Then $(e)$ holds.

Suppose that $u,v,w$ are all non-cut vertices in $G$. Suppose one of
$G-u-v$, $G-w-v$, $G-u-w$ is not connected, say $G-u-v$ is not
connected. Let $C_1,C_2,\ldots, C_r$ are connected components of
$G-u-v$. From $(d)$, there exists some connected component, say
$C_1$, such that $|V(C_1)|\leq g$. Clearly, $V(C_1)=\{w\}$. For each
$C_i \ (2\leq i\leq r)$, we have $|V(C_i)|\geq g+1$. Then $(g)$
holds. Suppose $G-u-v$, $G-w-v$, $G-u-w$ are all connected. Since
$G-\{x,y,z\}$ is not connected and each connected component of
$G-\{x,y,z\}$ has at least $g+1$ vertices, it follows that $(h)$
holds.

Conversely, we suppose that $(1)$ or $(2)$ or $(3)$ holds. Suppose
that $(1)$ holds. Since there exists a cut vertex set $\{u,v,w\}$ in
$G$ such that each connected component of $G-\{u,v,w\}$ has at least
$g+1$ vertices, it follows that $\kappa_g(G)\leq 3$. Since
$\kappa(G)=3$, it follows that $\kappa_g(G)\geq \kappa(G)=3$, and
hence $\kappa_g(G)=3$. Suppose that $(2)$ holds. From $(a)$, we have
$\kappa_g(G)\geq 3$. From $(b)$, we have $\kappa_g(G)\leq 3$, and
hence $\kappa_g(G)=3$. Suppose that $(3)$ holds. From $(c)$ and
$(d)$, we have $\kappa_g(G)\geq 3$. From $(e)$ or $(f)$ or $(g)$ or
$(h)$, we have $\kappa_g(G)\leq 3$, and hence $\kappa_g(G)=3$.
\end{proof}

\noindent {\bf Example 4.3.} Let $H_3$ be a graph obtained from
$K_{\lceil \frac{n-3}{2}\rceil}$ and $K_{\lfloor
\frac{n-3}{2}\rfloor}$ by adding two vertices $u,v,w$ and edges in
$\{uu_1,uu_2,uu_3,vv_1,vv_2,vv_3\}$, where $u_1,u_2,u_3\in
V(K_{\lceil \frac{n-3}{2}\rceil})$ and $v_1,v_2,v_3\in V(K_{\lfloor
\frac{n-3}{2}\rfloor})$. From Theorem \ref{th4-3},
$\kappa(H_3)=\kappa_g(H_3)=3$.\\

\noindent {\bf Example 4.4.} Let $H_4$ be a graph obtained from
$K_{\lceil \frac{n-3}{2}\rceil}$ and $K_{\lfloor
\frac{n-3}{2}\rfloor}$ by adding three vertices $u,v,x$ and then the
edges $ux,vx$, all the edges from $\{u,v\}$ to $K_{\lceil
\frac{n-3}{2}\rceil}$, and all the edges from $\{u,v\}$ to
$K_{\lfloor \frac{n-3}{2}\rfloor}$. From Theorem \ref{th4-3},
$\kappa(H_4)=2$
and $\kappa_g(H_4)=3$.\\

\noindent {\bf Example 4.5.} Let $H_5$ be a graph obtained from
$K_{\lceil \frac{n-3}{2}\rceil}$ and $K_{\lfloor
\frac{n-3}{2}\rfloor}$ by adding three vertices $v,x,y$ and then the
edges $vx,vy$, all the edges from $v$ to $K_{\lceil
\frac{n-3}{2}\rceil}$, and all the edges from $v$ to $K_{\lfloor
\frac{n-3}{2}\rfloor}$. From Theorem \ref{th4-3}, $\kappa(H_5)=1$
and $\kappa_g(H_4)=3$.

\section{Extremal problems}

We now consider the three extremal problems that we stated in the
Introduction.

Let $T_n'$ be a tree of order $n$ constructed as follows.

$(1)$ Let $K_{1,g}^1,K_{1,g}^2,\ldots,K_{1,g}^{r},K_{1,x}^{r+1}$ be
$(r+1)$ stars with centers $v_1,v_2,\ldots,v_{r+1}$, where $g+1\leq
x\leq 2g$ and $n-k=(g+1)r+x+1$;
\begin{figure}[!hbpt]
\begin{center}
\includegraphics[scale=0.7]{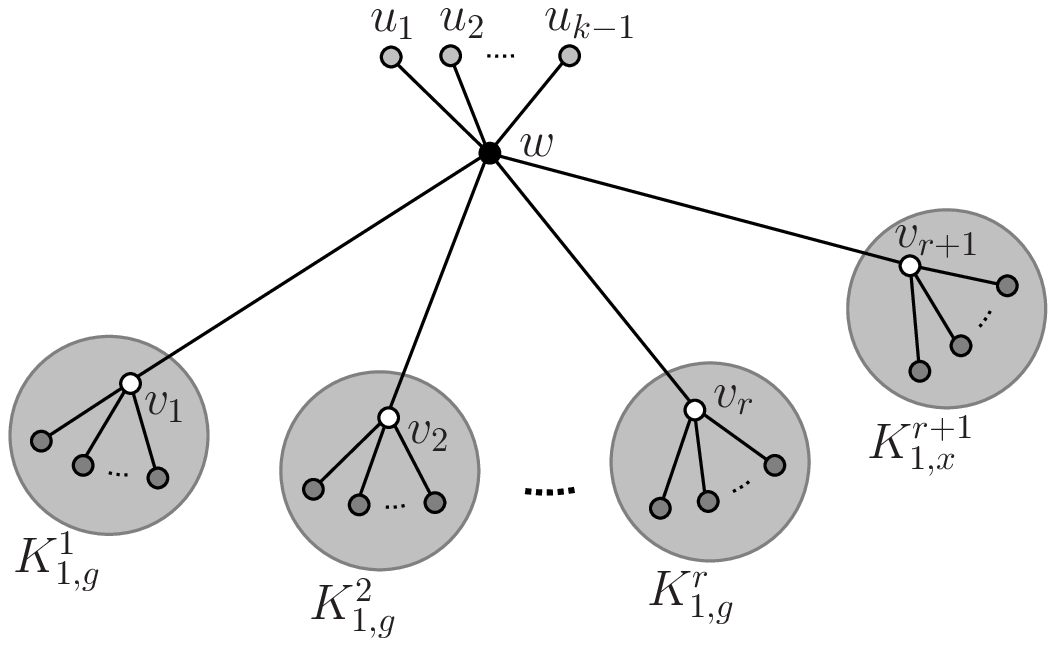}
\end{center}
\begin{center}
Figure 2: Tree $T_n'$.
\end{center}\label{fig1}
\end{figure}

$(2)$ Let $T_n'$ be a tree of order $n$ obtained from the stars
$K_{1,g}^1,K_{1,g}^2,\ldots,K_{1,g}^{r},K_{1,x}^{r+1}$ and the
vertices $w,u_1,u_2,\ldots,u_{k-1}$ by adding edges in
$\{wu_i\,|\,1\leq i\leq k-1\}\cup \{wv_i\,|\,1\leq i\leq r+1\}$; see
Figure $2$.

\begin{lemma}\label{lem5-1}
Let $n,g,k$ be three integers with $1\leq g\leq \left\lfloor
\frac{n-k-2}{2}\right\rfloor$. Then
$$
\kappa_g(T_n')=k.
$$
\end{lemma}
\begin{proof}
Choose $X=\{w\}\cup \{u_1,u_2,\ldots,u_{k-1}\}$. Clearly, $G-X$ is
not connected and each connected component of $G-X$ has at least
$g+1$ vertices, and hence $\kappa_g(T_n')\leq k$. It suffices to
show $\kappa_g(T_n')\geq k$. We only need to prove that for any
$X\subseteq V(T_n')$ with $|X|\leq k-1$, if $T_n'-X$ is not
connected, then there is a connected component of $T_n'-X$ having at
most $g$ vertices. Then we have the following claim.

\begin{claim}\label{Clm:8}
$w\in X$.
\end{claim}
\begin{proof}
Assume, to the contrary, that $w\notin X$. Since $T_n'-X$ is not
connected, it follows that there exists some center, say $v_1$, such
that $v_1\in X$. If there is an isolated vertex of $K_{1,g}^{1}-v_1$
in $G-X$, then we are done. Then all vertices of $K_{1,g}^{1}-v_1$
are not isolated vertices in $G-X$, and hence
$V(K_{1,g}^{1})\subseteq X$. Since $T_n'-X$ is not connected, it
follows that there exists some center, say $v_2$, such that $v_2\in
X-v_1$. If there is an isolated vertex of $K_{1,g}^{2}-v_2$ in
$G-X$, then we are done. We assume that $V(K_{1,g}^{2})\subseteq X$.
Continue this process, we have
$(\bigcup_{i=1}^{r+1}V(K_{1,g}^{i}))\subseteq X$. Clearly, $G-X$ is
connected, a contradiction.
\end{proof}

From Claim \ref{Clm:8}, we have $w\in X$. Since $|X|\leq k-1$, it
follows that there exists some $u_j$ such that $u_j\notin X$, and
hence $u_j$ is an isolated vertex in $G-X$. So there is a connected
component of $T_n'-X$ having at most $g$ vertices, and hence
$\kappa_g(T_n')\geq k$. So we have $\kappa_g(T_n')=k$.
\end{proof}

\begin{proposition}\label{pro5-1}
Let $n,g,k$ be three integers with $1\leq g\leq \left\lfloor
\frac{n-k-2}{2}\right\rfloor$ and $1\leq k\leq n-2g-2$. Then
$$
s(n,k)=n-1.
$$
\end{proposition}
\begin{proof}
Let $G=T_n'$. From Lemma \ref{lem5-1}, we have $\kappa_g(T_n')=k$,
and hence $s(n,k)\leq n-1$. Since we consider only connected graphs,
we have $s(n,k)\geq n-1$, and hence $s(n,k)=n-1$.
\end{proof}

\begin{lemma}\label{lem5-2}
Let $n,g,k$ be three integers with $1\leq g\leq \left\lfloor
\frac{n-k-2}{2}\right\rfloor$. Let $H_k$ be the graph obtained from
three cliques $K_{n-k-g},K_{k-1},K_{g+1}$ by adding the edges in
$E_{H_k}[K_{n-k-g},K_{k-1}]\cup E_{H_k}[K_{g+1},K_{k-1}]$. Then
$$
\kappa_g(H_k)=k-1.
$$
\end{lemma}
\begin{proof}
Let $X=V(K_{k-1})\subseteq V(G)$. Since $1\leq g\leq \left\lfloor
\frac{n-k-2}{2}\right\rfloor$, it follows that $G-X$ is not
connected and each component has at least $g+1$ vertices, and hence
$\kappa_g(H_k)\leq k-1$. Clearly, $\kappa_g(H_k)\geq
\kappa(H_k)=k-1$. So $\kappa_g(H_k)=k-1$.
\end{proof}

\begin{theorem}\label{th5-2}
Let $n,g,k$ be two integers with $1\leq g\leq \left\lfloor
\frac{n-k-2}{2}\right\rfloor$ and $1\leq k\leq n-2g-2$. Then
$$
f(n,k)={n\choose 2}-(n-k-g)(g+1)+1.
$$
\end{theorem}
\begin{proof}
To show $f(n,k)\geq {n\choose 2}-(n-k-g)(g+1)+1$, we construct $H_k$
defined in Lemma \ref{lem5-2}. Then $\kappa_g(H_k)=k-1$. Since
$|E(G_k)|={n\choose 2}-(n-k-g)(g+1)$, it follows that $f(n,k)\geq
{n\choose 2}-(n-k-g)(g+1)+1$.

Let $G$ be a graph with $n$ vertices such that $|E(G)|\geq {n\choose
2}-(n-k-g)(g+1)+1$. We claim that $\kappa_g(G)\geq k$. Assume, to
the contrary, that $\kappa_g(G)\leq k-1$. Then there exists a vertex
set $X\subseteq V(G)$ and $|X|\leq k-1$ such that each connected
component of $G-X$ has at least $g+1$ vertices. Let
$C_1,C_2,\ldots,C_r$ be the connected components of $G-X$. Clearly,
$|V(C_i)|\geq g+1$ for each $i \ (1\leq i\leq r)$. The number of
edges from $C_1$ to $C_2\cup C_3\cup \ldots \cup C_r$ in
$\overline{G}$ is at least $|V(C_1)|(n-|V(C_1)|-|X|)\geq
(n-k-g)(g+1)$. Clearly, $|E(G)|\leq {n\choose 2}-(n-k-g)(g+1)$,
which contradicts to $|E(G)|\geq {n\choose 2}-(n-k-g)(g+1)+1$. So
$\kappa_g(G)\geq k$, and hence $f(n,k)\leq {n\choose
2}-(n-k-g)(g+1)+1$.

From the above argument, we have $f(n,k)={n\choose
2}-(n-k-g)(g+1)+1$.
\end{proof}

\begin{proposition}\label{pro5-2}
Let $n,g,k$ be three integers with $k=n-2g-2$ and $g\geq 1$. Then
$$
g(n,k)={n\choose 2}-(g+1)^2.
$$
\end{proposition}
\begin{proof}
From Proposition \ref{pro1-3}, we have $g(n,k)\leq {n\choose
2}-(g+1)^2$. Let $F_k$ be the graph obtained from three cliques
$K_{n-2g-2},K_{g+1},K_{g+1}$ by adding the edges in
$E_{G_k}[K_{n-2g-2},K_{g+1}]\cup E_{G_k}[K_{g+1},K_{n-2g-2}]$. Then
$e(F_k)={n\choose 2}-(g+1)^2$ and $\kappa_g(F_k)\leq k=n-2g-2$, and
hence $g(n,k)\geq {n\choose 2}-(g+1)^2$. So $g(n,k)={n\choose
2}-(g+1)^2$.
\end{proof}

\begin{remark}\label{rem1-1}
Suppose $k<n-2g-2$. Let $T_n^*$ be the tree of order $n$ defined in
Lemma \ref{lem4-1}. Then $\kappa_g(T_n^*)=n-2g-2>k$ and
$e(T_n^*)=n-1$, and hence $g(n,k)\leq n-2$. Since we consider only
connected graphs, it follows that $g(n,k)\geq n-1$, a contradiction.
So $g(n,k)$ does not exist for $k<n-2g-2$.
\end{remark}

\section{Concluding Remark}

In this paper, we focus our attention on the $g$-extra connectivity
of general graphs. We have proved that $1\leq \kappa_g(G)\leq
n-2g-2$ for $0\leq g\leq \left\lfloor \frac{n-3}{2}\right\rfloor$.
Trees with $\kappa_g(T_n)=n-2g-2$ for general $g \ (0\leq g\leq
\left\lfloor \frac{n-3}{2}\right\rfloor)$ are characterized in this
paper. But the graphs with $\kappa_g(G)=n-2g-2$ is still unknown.
From Proposition \ref{pro3-1}, the classical $\kappa(G)$ is a
natural lower bound of $\kappa_g(G)$, but there is no upper bound of
$\kappa_g(G)$ in terms of $\kappa(G)$. From Proposition
\ref{pro3-2}, the classical $diam(G)$ is a natural upper bound of
$\kappa_g(G)$, but there is no lower bound of $\kappa_g(G)$ in terms
of $diam(G)$.

%%%%%%%%%%%%%%%%%%%%%%%%%%%%%%%%%%%%%%%%%%%%%%%%%%%%%%%%%%%%%%%%%%%%%%%%%%%%%%%%%%%%%%%%%%%%%%
%%%%%%%%%%%%%%%%%%%%%%%%%%%%%%%%%%%%%%%%%%%%%%%%%%%%%%%%%%%%%%%%%%%%%%%%%%%%%%%%%%%%%%%%%%%%%%
%%%%%%%%%%%%%%%%%%%%%%%%%%%%%%%%%%%%%%%%%%%%%%%%%%%%%%%%%%%%%%%%%%%%%%%%%%%%%%%%%%%%%%%%%%%%%%

\end{document}